\documentclass[11pt]{article}

\usepackage{mathrsfs, bm}

\usepackage[colorlinks,linkcolor={blue},citecolor={red}]{hyperref}
\everymath{\displaystyle}

\usepackage{amssymb,amsthm,amsmath,graphicx, mathabx}
\usepackage{color} 
\usepackage{enumerate}

\numberwithin{equation}{section}
\newtheorem{theorem}{Theorem}[section]
\newtheorem{lemma}[theorem]{Lemma}
\newtheorem{corollary}[theorem]{Corollary}
\newtheorem{proposition}[theorem]{Proposition}
\newtheorem{Lemma A.1}{Lemma A.1}
\theoremstyle{definition}
\newtheorem{definition}[theorem]{Definition}

\theoremstyle{remark}
\newtheorem{remark}[theorem]{Remark}

\allowdisplaybreaks[3]

\begin{document}

\title{A Geometric Characterization of Potential Navier-Stokes Singularities}


	\author{Zhen Lei\footnotemark[1] \and Xiao Ren\footnotemark[2] \and Gang Tian\footnotemark[3]}

\renewcommand{\thefootnote}{\fnsymbol{footnote}}
\footnotetext[1]{School of Mathematical Sciences, LMNS and Shanghai Key Laboratory for Contemporary Applied Mathematics,  Fudan University, China. Email: zlei@fudan.edu.cn}

\footnotetext[2]{Beijing International Center for Mathematical Research, Peking University, China. Email: renx@pku.edu.cn}
\footnotetext[3]{Beijing International Center for Mathematical Research, Peking University, China. Email: gtian@math.pku.edu.cn}

\maketitle

\begin{abstract}
  For a local suitable weak solution to the Navier-Stokes equations, we prove
  that if the vorticity vectors belong to a double cone in regions of high
  vorticity magnitude, then the solution is regular. Roughly speaking this implies that, near a potential singularity, the directions of vorticity cannot avoid any great circle on the unit sphere. Our method, based on
  the control of local vorticity fluxes, is inspired by the classical
  Kelvin-Helmholtz law for ideal fluids and the Type I regularity theory
  for axisymmetric Navier-Stokes solutions. 
\end{abstract}

\section{Introduction}\label{sec:1}

In this paper, we study the global regularity problem of the incompressible
Navier-Stokes equations
\begin{equation}
  \left\{ \begin{array}{l}
    \partial_t v - \Delta v + v \cdot \nabla v + \nabla p = 0,\\
    \nabla \cdot v = 0,
  \end{array} \right. \label{eq:NS}
\end{equation}
in three space dimensions, see {\cite{F}} for the precise statement. From
\eqref{eq:NS} one can derive that the vorticity $\omega = \nabla \times v$
solves
\begin{equation}
  \partial_t \omega - \Delta \omega + v \cdot \nabla \omega = \omega \cdot
  \nabla v. \label{eq:vorticity}
\end{equation}
The outstanding challenge of the global regularity problem of \eqref{eq:NS} is
usually attributed to the following two aspects. Firstly, the natural energy
estimates of \eqref{eq:NS} are supercritical with respect to the natural
scaling of \eqref{eq:NS}, making it hard to effectively control the
solution on fine scales. Secondly, there is a stretching mechanism for
$\omega$ due to the nonlinear term $\omega \cdot \nabla v$ in
\eqref{eq:vorticity}, which could potentially create large vorticity
magnitudes and lead to a finite-time blow-up. Note that both difficulties are
absent in two space dimensions.

Due to the supercritical barrier, most of the regularity results for
\eqref{eq:NS} proved so far involve some kind of \textit{critical/Type I}
assumption, that is, pointwise or integral estimates that are invariant
under the Navier-Stokes scaling
\[ v^{(\lambda)} = \lambda v (\lambda x, \lambda^2 t), \quad p^{(\lambda)} =
   \lambda^2 p (\lambda x, \lambda^2 t) . \]
A commonly used definition of Type I blow-ups can be found in Definition
\ref{def:Type-I}. It is still open whether such blow-ups could exist at all.
Usually, in addition to the Type I assumption, some extra smallness must be assumed or exploited to prove
regularity. For example, it is known that if certain scale-invariant
quantities are uniformly small, then the solution is regular, see
{\cite{CKN, TX, SS}}. In {\cite{ESS}}, Escauriaza,
Seregin and Sverak proved the $L^{\infty}_t L^3_x$ regularity criteria using
backward uniqueness theory for the heat operator, and the inherent local
smallness of the $L^3$ norm plays a crucial role. In {\cite{KNSS}}, Koch,
Nadirashvili, Seregin and Sverak developed the blow-up method for
\eqref{eq:NS} and tied the Type I regularity problem of \eqref{eq:NS} to
Liouville type conjectures for bounded mild ancient solutions. In the axially
symmetric case, Type I blow-ups have been successfully excluded, see
the works of Chen, Strain, Tsai and Yau \cite{CSYT1, CSYT2} as well as \cite{KNSS, S20}. Here, the
crucial extra smallness is gained by applying the De Giorgi-Nash-Moser theory to the equation
satisfied by the scalar function $\Gamma = r v_{\theta}$. Here, $r$
is the distance to the symmetry axis and $v_\theta$ is the angular component of velocity. We refer to {\cite{SS, LR}} and references
therein for more details on related
methods. From another viewpoint, these results provide characterizations of potential Navier-Stokes singularities  in terms of their blow-up rates.

Another approach in the regularity theory of Navier-Stokes equations relies on the study of the vortex stretching mechanism. For ideal fluids, the viscosity terms in \eqref{eq:NS} and
\eqref{eq:vorticity} are absent, and due to the Lie bracket structure of $v
\cdot \nabla \omega - \omega \cdot \nabla v$, vortex lines (integral curves of vorticity) move with the fluid. Moreover, the classical Kelvin-Helmholtz law states that the flux of vorticity through an evolving vortex tube
stays constant. In the presense of viscosity, the frozen-in property of
vorticity breaks down and the topology of the vortex lines could change \cite{ELP}.
However, certain cancellations between the advection term $v \cdot \nabla
\omega$ and the stretching term $\omega \cdot \nabla v$ may still be effective
to help prevent the formation of singularities. Indeed, the stabilizing effect
of the advection term $v \cdot \nabla \omega$ is known for equations with or without viscosity, see, \textit{e.g.}, {\cite{HL, HLi, Cjj}}.

 In {\cite{CF}}, Constantin and Fefferman discovered an interesting connection between the local alignment of vorticity
directions and the regularity of Navier-Stokes solutions. Denote the vorticity
direction vector by
\[ \xi := \frac{\omega}{| \omega |}, \]
which is well-defined if $| \omega |  \neq 0 $. Using an $L^2$
estimate for the vorticity equation \eqref{eq:vorticity}, \cite{CF} proved
the regularity of solutions assuming the spatial Lipshitz continuity of $\xi$ in regions
of high vorticity magnitude. To be precise, the condition reads that for any
pair of points $(x, t), (y, t)$, there holds either $\min (| \omega (x, t) |, | \omega (y, t) |) < M$ or
\begin{equation} 
  | \xi (x, t) \times \xi (y, t) | < C | x - y |. \label{eq:Lip}
\end{equation}
Here, $C$ and $M$ are arbitrary positive constants. The key tool in
{\cite{CF}} is the following singular integral representation for the
stretching factor:
\[ \xi \cdot \nabla v \cdot \xi = \frac{3}{4 \pi} P.V. \int D (\hat{y}, \xi (x
   + y), \xi (x)) | \omega (x + y) | \frac{d y}{| y |^3}, \]
   where
\[ \hat{y} := \frac{y}{| y |}, \quad D (a, b, c) := (a \cdot c) \, \mbox{Det} (a,
   b, c) . \]
Later, based on similar ideas, Beirao da Veiga and Berselli {\cite{BB}}
weakened the condition \eqref{eq:Lip} to $\frac12$-H\"older continuity, that is,
\begin{equation}
  | \xi (x, t) \times \xi (y, t) | < C | x - y |^{\frac{1}{2}} .
  \label{eq:1-2-Holder}
\end{equation}
In {\cite{GM}}, Giga and Miura considered $L^{\infty}_x$ mild solutions
satisfying a pointwise Type I bound, and using a blow-up procedure they proved
that for such solutions the uniform spatial continuity of $\xi$ (again, in regions of high vorticity magnitude) would guarantee regularity. See also
{\cite{BGG, BP, Gsurvey}} and references therein for more results and discussions on geometric regularity criteria.

In this paper, we give a new geometric characterization of potential
Navier-Stokes singularities based on the vorticity directions. Instead of \eqref{eq:Lip}, \eqref{eq:1-2-Holder}, we work with scale-invariant scenarios concerning the range of $\xi$. Our method is
very different from those in {\cite{CF, BB} or \cite{GM}}, in particular, it relies heavily on the Lie bracket structure of the
nonlinear terms in \eqref{eq:vorticity}.

The first result deals with the case when $\omega$  belongs to either a
double cone or a (large) ball. It should be stressed that we do not need any smallness
assumption on the opening angle of the cone. For the sake of generality, we
formulate the regularity criterion for local suitable weak solutions, which
include the class of	 mild solutions with finite local energies. Let $Q(r) = B(r) \times (-r^2, 0)$ be the standard parabolic cylinders centered at the origin.

\begin{theorem}
  \label{thm:main}Let $v$ be a suitable weak solution to the Navier-Stokes
  equations in $Q(1)$. Suppose that there exists a unit vector $e \in \mathbb{S}^2$
  and some positive constants $\delta, M$ such that, for any regular point \footnote{We say that $(x,t)$ is a regular point of $v$ if $v$ is bounded in a parabolic cylinder centered at $(x,t)$.} $(x,t) \in Q(1)$ of $v$, there holds either $|\omega(x,t)| \le M$ or
  \begin{equation}
    | \xi (x, t) \times e | \le 1 - \delta, \label{eq:cone}
  \end{equation}
  then $v$ is regular in $\overline{Q(1/2)}$.
\end{theorem}

\begin{remark}
By rotational invariance of the Navier-Stokes equations, we may restrict our attention to
the case $e = e_3 = (0, 0, 1)$. In this case, the assumption of Theorem \ref{thm:main} is
equivalent to
\[ \sqrt{\omega_1^2 + \omega_2^2} =: | \omega_h | \le C | \omega_3 | \]
{\large{(}}with $C = \frac{1 - \delta}{\sqrt{2\delta -
\delta^2}}${\large{)}} for any regular point $(x, t) \in Q (1)$ with $| \omega
(x, t) | > M$. Hence, to prove Theorem \ref{thm:main}, it is sufficient to
work under the condition that
\begin{equation}
  | \omega | \le C | \omega_3 | + M \label{eq:cone-3}
\end{equation}
at all regular points.
\end{remark}

\begin{remark}
	In view of Theorem \ref{thm:main}, it is natural to ask what happens if the vorticity is contained in a half space (a cone with opening angle $\pi$), that is,
	\begin{equation} \label{eq:115-6}
		\omega_3 \ge 0.
	\end{equation}
	In this case, it seems nontrivial to prove regularity even for axisymmetric solutions. Note that the classical regularity result for axisymmetric flows with no swirl \cite{UI} is a special case of  \eqref{eq:115-6}  with $\omega_3 \equiv 0$. 
\end{remark}

\begin{remark}
	Proving a similar theorem for the direction of velocity seems more challenging. We give a simple result in this direction for axisymmetric solutions in Section \ref{sec:last}.
\end{remark}

Theorem \ref{thm:main} can be viewed as a geometric characterization of
potential singularities of the Navier-Stokes equations. Let $v$ be a suitable weak Navier-Stokes solution in $Q(1)$, and $\mathcal{I}
\subset \mathbb{S}^2$ be the set defined as
\[ \mathcal{I} := \bigcap_{0<r<1} \bigcap_{M > 0} \overline{\mathcal{I}_{r, M}} \]
with
\[ \mathcal{I}_{r, M} := \bigg\{ \xi (x, t) : \, v
   \ \mbox{is regular at} \ (x, t) \in Q (r), \, | \omega (x, t) | > M
   \bigg\} . \]
Here, $\overline{\mathcal{I}_{r, M}}$ stands for the topological closure of $\mathcal{I}_{r, M}$ in
$\mathbb{S}^2$. If the origin is a regular point of $v$, then it is clear that $\mathcal{I} = \emptyset$. Intuitively, $\mathcal{I}$ should be a sufficiently ``large''
set if blow-up occurs at the origin. Indeed, Theorem \ref{thm:main} directly implies
that 

\begin{corollary} \label{cor:114}
	$\mathcal{I}$ intersects with every great circle on $\mathbb{S}^2$ 
	if and only if the origin is a singular point.
\end{corollary}
 
 This statement can be compared
with the classical value distribution theory of the Gauss map of minimal
surfaces, where much stronger results are available, see, e.g., {\cite{AFF}}.

\smallskip

The next result concerns the angle between the vorticity vectors within each
time slice. In contrast to the conditions \eqref{eq:Lip} and
\eqref{eq:1-2-Holder}, we do not require smallness assumptions on such angles to ensure regularity.
Instead, it is only assumed that such angles are uniformly separated from ${\frac{{\pi}}{2}}$.

\begin{corollary}
  \label{thm:rot}Let $v$ be a suitable weak solution to the Navier-Stokes
  equations in $Q (1)$. Suppose that for some positive constant $\delta$, we
  have
  \begin{equation}
    | \xi (x, t) \times \xi (y, t) | < 1 - \delta \label{eq:cone-t}
  \end{equation}
  for any pair of regular points $(x, t), (y, t) \in Q (1)$ with   $| \omega (x, t) | \cdot | \omega
  (y, t) | \neq 0$, then $v$ is regular at the origin.
\end{corollary}

Corollary \ref{thm:rot} is not a full improvement over the results of
{\cite{CF}} and {\cite{BB}}, because here we have to assume \eqref{eq:cone-t}
throughout the region where $| \omega | > 0$, instead of $| \omega | > M$.  We cannot remove this seemingly technical restriction due to the lack of effective control on the geometry of regions of
small vorticity magnitude.


\paragraph{Main ideas of the proof.} We start with the
observation that, under the condition
\begin{equation}
  | \omega | \leqslant C | \omega_3 |, \label{eq:888}
\end{equation}
the flipped vorticity field
\begin{equation}
  \tilde{\omega} = \omega \, \mbox{sgn} \omega_3 \label{eq:mod-vorticity}
\end{equation}
is divergence free. Hence, we can uniformly control the local flux of
$\tilde{\omega}$ through horizontal planes, which is equal to the local
integral of $| \omega_3 |$ over horizontal variables. This quantity will be
referred to as the \textit{absolute vorticity flux}. Remarkably, together
with the condition \eqref{eq:888}, local uniform boundedness of this quantity
gives an $L^{\infty}_t L^{\infty}_{x_3} L^1_{x_1, x_2}$ estimate on $\omega$,
which is \textit{critical} with respect to the Navier-Stokes scaling. In
particular, it allows one to derive uniform Type I bounds in the sense of
Definition \ref{def:Type-I}.

Under the weaker condition \eqref{eq:cone-3}, the flipped vorticity
\eqref{eq:mod-vorticity} is no longer divergence free. Nonetheless, we can
still prove uniform boundedness of the absolute vorticity flux using the time
evolution of the flux of the modified vorticity
\[ {\widehat{\omega}} = \frac{\omega \omega_3}{\sqrt{\omega_3^2 + 1}},
\]
which is ``almost'' divergence free. Note that we rely on the
cancellations between the nonlinear terms in \eqref{eq:vorticity} which are
responsible for the validity of the Kelvin-Helmholtz law in the inviscid
case. To prove Theorem \ref{thm:main}, we still need to find some additional
smallness to break the critical scaling. Inspired by the axisymmetric Type I
regularity theory \cite{CSYT1, CSYT2, KNSS}, we prove a local decay property of the absolute vorticity
flux with the help of De Giorgi-Nash-Moser theory. Here, the quantitative
 intervals of regularity constructed in {\cite{LR}} turn out to be very useful. Finally, we prove Corollary  \ref{cor:114} and Corollary  \ref{thm:rot} by reducing them to the case considered
in Theorem \ref{thm:main}.

Unlike 2D flows, the scenario considered by Theorem \ref{thm:main} allows the amplification of vorticity by the vortex stretching mechanism. At a very heuristic level our proof suggests that the stretching of vortex tubes alone may be insufficient to generate a singularity; the twisting and folding of vortex tubes, similar to those considered in dynamo theory \cite[Section V.4]{AK}, could play an important role as it enhances the local vorticity fluxes. It is interesting to ask whether similar results hold for the Euler equations as well. We mention that for the 1D De Gregorio model, the sign of ``vorticity" indeed makes a difference in global regularity versus singularity formation \cite{HTW, LLR}. 

\paragraph{Organization of the paper.} In Section \ref{sec:2}, we
collect some useful notations, definitions and known results, which will be useful in the proof. In Section
\ref{sec:3}, we prove the local uniform boundedness of the absolute vorticity
flux under the conditions \eqref{eq:888} or \eqref{eq:cone-3}. As a consequence, we derive the Type I
bounds and then carry out the standard blow-up procedure at the potential singularity. In Section
\ref{sec:4}, we prove a local decay property of the absolute vorticity flux for
the limiting ancient solution, which finishes the proof of Theorem
\ref{thm:main}. In Section \ref{sec:5}, we give the proofs of  Corollary  \ref{cor:114} and Corollary \ref{thm:rot}. Finally, in Section \ref{sec:last} we give a simple result concerning the direction of \emph{velocity} in the axisymmetric case. 

\section{Preliminaries}\label{sec:2}

Given $x = (x_1, x_2, x_3) \in \mathbb{R}^3$, we often write $x_h = (x_1,
x_2)$. Due to the special role played by the $x_3$ direction in
condition \eqref{eq:cone-3}, sometimes it is convenient to use cylinders
instead of balls. Hence, we introduce the notations
\[ B (r) = \{ x \in \mathbb{R}^3 : | x | < r \}, \quad Q (r) = B (r) \times (-
   r^2, 0), \]
\[ \mathcal{B} (r) = \{ x \in \mathbb{R}^3 : | x_h | < r, | x_3 | < r \}, \quad
   \mathcal{Q} (r) = \mathcal{B} (r) \times (- r^2, 0), \]
and write their translated versions as
$$B(r;x_0, t_0) = \{(x+x_0, t_0): x \in B(r)\}, \ Q(r;x_0, t_0) = \{(x+x_0, t+t_0): (x,t) \in Q(r)\}, \ \emph{etc}.$$
Moreover, we will make frequent use of the following discs and circles:
\[ \mathcal{D} (r, z, t) = \{ (x_h, x_3, t) : | x_h | < r, x_3 = z \}, \quad \mathcal{S}
   (r, z, t) = \{ (x_h, x_3, t) : | x_h | = r, x_3 = z \}. \]
   In the proof, we write $C$ for positive constants whose values are not necessarily fixed from line to line.

   \smallskip
\begin{remark}
  In this section, we state all the definitions and lemmas using $B (r)$ and
  $Q (r)$, but one can directly replace them with $\mathcal{B} (r)$ and
  $\mathcal{Q} (r)$ when the latter choice is more convenient. Of course, one
  can also make such replacements in Theorems \ref{thm:main} and Corollary \ref{thm:rot}
  without changing their strength at all.
\end{remark}

\subsection{Suitable weak solutions}\label{sec:21}

Throughout our proof, $(v, p)$ will be a local suitable weak solution to the
Navier-Stokes equations in $Q (1)$ or $\mathcal{Q} (1)$. By definition, a
local suitable weak solution in $Q (1)$ satisfies the following conditions.
\begin{enumerate}[(1)]
  \item Finite local energies.
  \begin{equation}
    \sup_{- 1 < t < 0} \int_{B (1)} | v |^2 d x + \int_{Q (1)} | \nabla v |^2
    d x d t + \int_{Q (1)} | p |^{\frac{3}{2}} d x d t < + \infty
    \label{eq:finite-local-energy}
  \end{equation}
  \item $(v, p)$ solves the Navier-Stokes equations in the sense of
  distributions.
  
  \item Local energy inequality. For any nonnegative smooth function $\phi$
  compactly supported in $Q (1)$ and any $- 1 < t < 0$,
  \[ \int_{B (1)} | v (x, t) |^2 \phi (x, t) d x + 2 \int_{- 1}^t \int_{B
     (1)} | \nabla v |^2 \phi d x d t \leqslant \]
  \begin{equation}
    \leqslant \int_{- 1}^t \int_{B (1)} [| v |^2 (\partial_t \phi + \Delta
    \phi) + (| v |^2 + 2 p) v \cdot \nabla \phi] d x d t
    \label{eq:local-energy-ineq}
  \end{equation}
\end{enumerate}
We shall often write
\[ G = G [v, p] := \int_{Q (1)} \left( | v |^3 + | p |^{\frac{3}{2}}
   \right) d x d t + 2, \]
which is finite due to \eqref{eq:finite-local-energy} and interpolation. The constant $2$ in the definition of $G$ is added to make the statement of Lemma \ref{thm:qrs} below more convenient. The
existence of global suitable weak solutions starting from an arbitrary $L^2$
initial data is classical (See {\cite{CKN}}). In the local regularity theory of
suitable weak solutions, a number of scale-invariant quantities often play an
important role. We denote
\begin{equation}
  F (r) = \frac{1}{r^2} \int_{Q (r)} | v |^3 d x d t, \quad E (r) =
  \frac{1}{r} \int_{Q (r)} | \nabla v |^2 d x d t \label{eq:FE}
\end{equation}
\begin{equation}
  A (r) = \frac{1}{r} \sup_{- r^2 < t < 0} \int_{B (r)} | v |^2 d x, \quad D
  (r) = \frac{1}{r^2} \int_{Q (r)} | p |^{\frac{3}{2}} d x d t \label{eq:AD}
\end{equation}
We also define these scale-invariant quantities centered at an arbitrary point
$(x_0, t_0)$, for example, let
\begin{equation} \label{eq:0115-1} F (r; x_0, t_0) = \frac{1}{r^2} \int_{Q (r)} | v (x_0 + x, t_0 + t) |^3 d x
   d t, \quad \emph{\mbox{etc}.} 
   \end{equation}
\begin{definition}
  \label{def:Type-I}We say that $(v, p)$ satisfies the Type I condition at the origin
  if the scale-invariant quantities are uniformly bounded,
  \emph{i.e.},
  \begin{equation}
    \sup_{0<r<1} \, (F + D + E + A) (r) < + \infty .
    \label{eq:Type-I-condition}
  \end{equation}
  We say that $(v, p)$ satisfies the \emph{uniform} Type I condition in $Q \bigg(
  \frac{1}{2} \bigg)$ if
  \begin{equation}
    \sup_{(x_{0,} t_0) \in Q \left( \frac{1}{2} \right)} \sup_{0 < r <
    \frac{1}{2}} (F + D + E + A) (r; x_0, t_0) < + \infty .
    \label{eq:unif-Type-I}
  \end{equation}
\end{definition}

Note that \eqref{eq:Type-I-condition} and \eqref{eq:unif-Type-I} are implied by various types of critical assumptions, see \cite{SS, SZ}. A standard blow-up procedure  can be
applied to a potential singularity satisfying the (uniform) Type I condition, see, \emph{e.g.},
{\cite[Proposition 2]{SS}}. The resulting limiting solution will be a suitable
weak ancient solution living on $\mathbb{R}^3_x \times (- \infty, 0)_t$, which
satisfies the (uniform) Type I estimates in a global sense.

\smallskip

In Section \ref{sec:3}, we will need the following lemma to derive uniform Type I estimates using the absolute vorticity flux.

\begin{lemma}
  \label{lem:critical-Type-I}Let $\frac{3}{2} < q < 2$ be fixed. Suppose that
  \[ \Lambda_q := \sup_{r > 0} \sup_{- r^2 < t < 0} r^{q - 3} \int_{B
     (r)} | v |^q d x < + \infty, \]
  then $(v, p)$ satisfies Type I condition at the origin:
  \[ \limsup_{r \rightarrow 0 +} (F + D + E + A) (r) < C (q, \Lambda_q) . \]
  Moreover, $C (q, \Lambda_q) \rightarrow 0$ if $\Lambda_q \rightarrow 0$.
\end{lemma}

\begin{proof}
  This is a corollary of {\cite[Theorem 6]{SS}}, by taking $s = q$ and $\frac{2q}{2q
  - 3} < l < \frac{3q}{2q - 3}$ there. 
\end{proof}

\subsection{Quantitative partial regularity}\label{sec:22}

The classical Caffarelli-Kohn-Nirenberg theorem {\cite{CKN}} states that the
1D parabolic Hausdorff measure of the singular set of a suitable weak solution to \eqref{eq:NS}
must be zero. A particular corollary is the existence of regular (spatial or
temporal) intervals for suitable weak solutions. Indeed, it is not difficult
to show that for a suitable weak solution $v$ in $Q (1)$, there exist numbers
$$a \in \left( \frac{2}{3}, \frac{3}{4} \right),\quad  \delta \in \left( 0,
\frac{1}{10} \right)$$ 
such that $v$ is regular in the space-time closure of the parabolic
shell $Q (a + \delta) \setminus Q (a - \delta)$. However, the
Caffarelli-Kohn-Nirenberg theorem does not give an explicit lower bound on
$\delta$, nor any regularity estimates of $v$ in $Q (a + \delta) \setminus Q
(a - \delta)$.

Recently, a quantitatively improved partial
regularity theorem, along with the existence of quantitative intervals of regularity, is established in {\cite{LR}} using an energy pigeonholing
argument. This result will be useful for us in proving the local smallness of
vorticity fluxes, see Section \ref{sec:4}.

\begin{lemma}
  \label{thm:qrs}({\cite{LR}} Qantitative shells of regularity) Let $v$ be a
  suitable weak solution in $Q (1)$ and denote
  \[ G = \int_{Q (1)} \left( | v |^3 + | p |^{\frac{3}{2}} \right) d x d t + 2
     < + \infty . \]
  There exist numbers $a \in \left( \frac{2}{3}, \frac{3}{4} \right)$ and
  $\delta \in \left( G^{- O (G)}, \frac{1}{10} \right)$ such that $v$ is regular
  in the space-time closure of $Q (a + \delta) \setminus Q (a - \delta)$ with
  the estimates
  \begin{equation}
    \sup_{Q (a + \delta) \setminus Q (a - \delta)} (| v | + | \nabla v | + |
    \nabla^2 v |) < G^{O (G)} . \label{eq:reg-est-shell}
  \end{equation}
  Here, $O(G)$ stands for a positive number bounded by $G$ times a universal constant.
\end{lemma}

To be precise, quantitative intervals of regularity  in the spatial direction is stated
in {\cite[Theorem B]{LR}}, and the corresponding (but much easier) result in the
temporal direction is stated in {\cite[Proposition 9]{LR}}. Combining these
results, it is easy to deduce Lemma \ref{thm:qrs}. Moreover, by simple
modifications in the proof (see {\cite[Remark 8]{LR}}), one can replace $B
(r), Q (r)$ with $\mathcal{B} (r), \mathcal{Q} (r)$ in the statement of Lemma \ref{thm:qrs}.


\subsection{A lemma from De Giorgi-Nash-Moser theory}\label{sec:23}

\begin{lemma}
  \label{lem:dnm}Suppose $V (x, t)$ is a nonnegative Lipshitz supersolution in
  $B (1) \times (0, T)$ to the equation
  \begin{equation}
    \partial_t V - \Delta V + b \cdot \nabla V = 0 \label{eq:drift-diffusion}
  \end{equation}
  where the drift $b \in L^{\infty} (B (1) \times (0, T))$ and $\nabla \cdot b
  = 0$. Assume that
  \[ \mbox{meas}\, (\{ x \in B (1) : V (x, \bar{t}) \geqslant \lambda \})
     \geqslant \delta \]
  for some $\bar{t} \in \left( 0, \frac{T}{3} \right)$, $\lambda > 0$ and
  $\delta > 0$, then given any $0 < r < 1$ there exists a constant $\beta =
  \beta (\delta, T, r, \| b \|_{L^{\infty} (B (1) \times (0, T))}) > 0$ such
  that
  \[ V \geqslant \beta  \lambda \]
  in $B (r) \times \left( \frac{T}{2}, T \right)$.
\end{lemma}

The exact choice of domains in Lemma \ref{lem:dnm} can be rather arbitrary, as long as the relation $\bar{t}<\frac{T}{2}$ is respected. For the proof of this lemma, we refer to {\cite[Corollary 3.3]{NU}}. In
Section \ref{sec:4}, we shall apply it to the axi-symmetric function $V =
\sup_{\mathcal{Q} \left( \frac{9}{10} \right)} \Gamma - \Gamma$ where
\[ \Gamma (r, z, t) = \int_{\mathcal{D} (r, z, t)} | \omega_3 (x_h, z, t) | d
   x_1 d x_2 . \]
It turns out, under the geometric condition \eqref{eq:888} on the
vorticity direction, $\Gamma$ is a subsolution to a drift-diffusion equation
with suitable drift (see \eqref{eq:Gamma-34}) inside the shells of regularity  given
by Lemma \ref{thm:qrs}.

\section{Boundedness of the vorticity flux}\label{sec:3}

In this section, we obtain the uniform boundedness of the local absolute
vorticity fluxes using the assumption \eqref{eq:cone-3}. Let $v$ be a local
suitable weak solution to \eqref{eq:NS} in $\mathcal{Q} (1)$. Here, we shall use the definitions and lemmas from Section \ref{sec:2} with cylindrical
domains $\mathcal{B} (r)$, $\mathcal{Q} (r)$ instead of $B (r), Q (r)$.

First of all, by a simple reduction, in Theorem \ref{thm:main} we can assume
without loss of generality that $v$ is regular for all negative times. Indeed,
let $a, \delta>0$ be the numbers guaranteed by Lemma \ref{thm:qrs} so that $v$
is regular in $\overline{\mathcal{Q} (a + \delta) \setminus \mathcal{Q} (a -
\delta)}$. If $v$ is regular in $\overline{\mathcal{Q} (a)}$, then there is
nothing to prove. Otherwise, let $(x_{\ast}, t_{\ast})$ be one of the earliest
blow-up points of $v$ in $\mathcal{Q} (a)$. The translated and rescaled
solution
\[ \tilde{v} (x, t) = \frac{\delta}{10} v \left( \frac{\delta x}{10}  +
   x_{\ast}, \frac{\delta^2 t}{100}  + t_{\ast} \right) \]
would still satisfy the conditions of Theorem \ref{thm:main} and, by construction, it
is regular for all negative times. Then,  it suffices to show the regularity of
$\tilde{v}$ at the origin to obtain a contradiction.

In the sequel, we always assume that $v$ is regular for all negative times. Let us start with a simple observation in a special case, which will be
important for Section \ref{sec:5}. Assume for now that \eqref{eq:cone-3}
holds with $M = 0$, that is, for any point in $\mathcal{Q}(1)$,
\begin{equation}
  | \omega | \leqslant C | \omega_3 | . \label{eq:cone-3-0}
\end{equation}
We define the flipped vorticity $\tilde{\omega} = \frac{\omega_3}{| \omega_3 |} \omega$
and its approximations $\tilde{\omega}^{(\epsilon)} = \frac{\omega_3}{\sqrt{\omega_3^2 +
\epsilon^2}} \omega$. By construction, at a fixed negative time $\tilde{\omega}^{(\epsilon)}
\rightarrow \tilde{\omega}$ uniformly as $\epsilon \rightarrow 0$. One can explicitly compute
\[ \nabla \tilde{\omega}^{(\epsilon)} = \frac{\omega_3}{\sqrt{\omega_3^2 +
   \epsilon^2}} \nabla \omega + \nabla \frac{\omega_3}{\sqrt{\omega_3^2 +
   \epsilon^2}} \otimes \omega \]
\[ = \frac{\omega_3}{\sqrt{\omega_3^2 + \epsilon^2}} \nabla \omega +
   \frac{\epsilon^2}{(\omega_3^2 + \epsilon^2)^{\frac{3}{2}}} \nabla \omega_3
   \otimes \omega, \]
and observe that, as $\epsilon \rightarrow 0$,
\[ \nabla \tilde{\omega}^{(\epsilon)} \rightarrow \frac{\omega_3}{| \omega_3
   |} \nabla \omega \]
on the set  \{$| \omega_3 | \neq 0$\}. Due to \eqref{eq:cone-3-0}, we have
$$\nabla \tilde{\omega}^{(\epsilon)} \in L^\infty_{\text{loc}}(\mathcal{Q}(1)),$$
and
\[ \nabla \tilde{\omega}^{(\epsilon)} = 0 \]
on the set $\{ | \omega_3 | = 0 \}$. Taking the limit $\epsilon \to 0$, we obtain that 
\[ \nabla \tilde{\omega} = {{\mathbf{1}_{\{ | \omega_3 | \neq 0 \}}} } 
    \frac{\omega_3}{| \omega_3 |} \nabla \omega, \]
in the weak sense and, in particular,
\begin{equation}
  \nabla \cdot \tilde{\omega} = {{\mathbf{1}_{\{ | \omega_3 | \neq 0 \}}} } 
   \frac{\omega_3}{| \omega_3 |} \nabla \cdot \omega = 0.
  \label{eq:tilde-div-free}
\end{equation}
Let $a, \delta>0$ be the numbers guaranteed by Lemma \ref{thm:qrs} for our
solution $v$. Using \eqref{eq:tilde-div-free} we deduce that, uniformly in $t < 0$ and $z
\in (- a, a)$,
\begin{eqnarray*}
  \int_{\mathcal{D} (a, z, t)} | \omega_3 | d x_1 d x_2 & = &
  \int_{\mathcal{D} (a, z, t)} \tilde{\omega}_3 d x_1 d x_2\\
  & = & \int_{\mathcal{D} (a, a, t)} \tilde{\omega}_3 d x_1 d x_2 + \int_z^a
  \int_{\mathcal{S} (a, \zeta, t)} \tilde{\omega} \cdot e_r d s d \zeta\\
  & \leqslant & G^{O(G)},
\end{eqnarray*}
where we have used \eqref{eq:reg-est-shell} in the last step. Hence, using
\eqref{eq:cone-3-0} again we obtain a uniform bound on $\omega$ given by
\begin{equation}
  \sup_{- \frac{1}{4} < t < 0} \sup_{- \frac{1}{2} < z < \frac{1}{2}}
  \int_{\mathcal{D} \left( \frac{1}{2}, z, t \right)} | \omega | d x_1 d x_2
  \leqslant G^{O(G)}. \label{eq:unif-omega}
\end{equation}
Note that the left hand side of \eqref{eq:unif-omega} is dimensionless with
respect to the Navier-Stokes scaling.

Next, consider the general case: \eqref{eq:cone-3} wih $M > 0$. The above
argument breaks down due to the lack of effective estimates for regions with
small vorticity magnitude. Nonetheless, we shall obtain the same critical
bound \eqref{eq:unif-omega} using the vorticity equation \eqref{eq:vorticity}
and a variation of the above observation \eqref{eq:tilde-div-free}. For this
purpose, structural cancellations between the advection term $v \cdot \nabla
\omega$ and the stretching term $\omega \cdot \nabla v$ in
\eqref{eq:vorticity} will be essential.

\begin{lemma}
  \label{lem:flux}Let  ($v$,p)  be a suitable weak solution to the
  Navier-Stokes equations in $\mathcal{Q} (1)$, regular for negative times, satisfying \eqref{eq:cone-3},
  then \eqref{eq:unif-omega} holds.
\end{lemma}

\begin{proof}
  Denote $f (s) = \sqrt{s^2 + 1}$ and note that $f'(s) = \frac{s}{f(s)}$, $f''(s) = \frac{1}{f(s)^3}$. Condition \eqref{eq:cone-3} implies that
  \begin{equation}
    | \omega | \leqslant C f (\omega_3) \label{eq:cone-f}
  \end{equation}
  for some positive constant $C$. Let $a, \delta>0$ be the numbers guaranteed by
  Lemma \ref{thm:qrs}. By \eqref{eq:cone-3} it suffices to show that
  \[ \sup_{t \in (- a^2, 0)} \sup_{z \in (- a, a)} \int_{\mathcal{D} (a, z,
     t)} f (\omega_3) d x_1 d x_2 < + \infty . \]
  It is well-known that $\omega_3$ satisfies the vorticity equation
  \begin{equation}
    \partial_t \omega_3 - \Delta \omega_3 + v \cdot \nabla \omega_3 - \omega
    \cdot \nabla v_3 = 0. \label{eq:omega-3}
  \end{equation}
  Multiplying \eqref{eq:omega-3} by $\frac{\omega_3}{f (\omega_3)}$, we get
  \begin{equation}
    \partial_t f (\omega_3) - \Delta [f (\omega_3)] + \frac{| \nabla \omega_3
    |^2}{f (\omega_3)^3} + v \cdot \nabla [f (\omega_3)] - \frac{\omega_3
    \omega}{f (\omega_3)} \cdot \nabla v_3 = 0. \label{eq:f-omega-3}
  \end{equation}
  We introduce
  \[ W (z, t) = \int_{\mathcal{D} (a, z, t)} f (\omega_3) d x_1 d x_2, \]
  and
  \[ \mathcal{P}_a = \{ (z, t) : - a < z < a, - a^2 < t < 0 \} . \]
  Integrating \eqref{eq:f-omega-3} in $\mathcal{D} (a, z, t)$ for $(z, t) \in
  \mathcal{P}_a$, we get
  \[ \partial_t W - \partial_z^2 W + {\int_{\mathcal{D} (a, z, t)}}  \frac{|
     \nabla \omega_3 |^2}{f (\omega_3)^3} d x_1 d x_2 \]
  \begin{equation} \label{eq:114-1} \leqslant \int_{\mathcal{D} (a, z, t)} \left\{ - v \cdot \nabla [f
     (\omega_3)] + \frac{\omega_3 \omega}{f (\omega_3)} \cdot \nabla v_3
     \right\} d x_1 d x_2 + B_1 \end{equation}
     Here,  $B_1$ stands for the boundary term
       \[  \int_{\mathcal{S} (a, z, t)} \partial_r [f (\omega_3)] d s, \]
       with $\partial_r = e_r \cdot \nabla$ and $e_r = \frac{x_1 \partial_1 + x_2 \partial_2}{| x_h
       	|}$.   For the integral in \eqref{eq:114-1}, we perform integration by parts for the products of horizontal components while keeping the vertical ones, and deduce that
  \begin{align} \eqref{eq:114-1}&= \int_{\mathcal{D} (a, z, t)} \left\{ (\nabla_h \cdot v_h) f (\omega_3)
     - v_3 \partial_3 f (\omega_3) - \left( \nabla_h \cdot \frac{\omega_3
     \omega_h}{f (\omega_3)} \right) v_3 + \frac{\omega_3^2}{f (\omega_3)}
     \partial_3 v_3 \right\} d x_1 d x_2 \nonumber \\
    &\quad  + B_1 + B_2 \nonumber \\
  &= \int_{\mathcal{D} (a, z, t)} \left\{ - \partial_3 v_3 f (\omega_3) -
     \frac{v_3 \omega_3 \partial_3 \omega_3}{f (\omega_3)} + \left( -
     \frac{\omega_h \cdot \nabla_h \omega_3}{f (\omega_3)^3} + \frac{\omega_3
     \partial_3 \omega_3}{f (\omega_3)} \right) v_3 + \frac{\omega_3^2}{f
     (\omega_3)} \partial_3 v_3 \right\} d x_1 d x_2 \nonumber \\
  &\quad  + B_1 + B_2 \nonumber \\
    &= \int_{\mathcal{D} (a, z, t)} \left\{ - \frac{\partial_3 v_3}{f
    (\omega_3)} - \left( \frac{\omega_h \cdot \nabla_h \omega_3}{f
    (\omega_3)^3} \right) v_3 \right\} {dx}_1 {dx}_2 + B_1 + B_2.
    \label{eq:W-rhs}
  \end{align}
  Here, $B_2$ stands for the boundary term
  \[ \int_{\mathcal{S} (a, z, t)} \left\{ - f (\omega_3) v_h \cdot e_r  +
     \frac{v_3 \omega_3 \omega_h \cdot e_r}{f (\omega_3)} \right\} d s. \]
  Note that $B_1$ and $B_2$ can be bounded pointwisely for $(z, t) \in
  \mathcal{P}_a$ using Lemma \ref{thm:qrs}, namely,
  \begin{equation}
    | B_1 | + | B_2 | \leqslant G^{O(G)} . \label{eq:B1-B2}
  \end{equation}
  Using \eqref{eq:cone-f} and Cauchy-Schwarz inequality, the second part of the integral in
  \eqref{eq:W-rhs} can be bounded as
  \begin{equation}
    \int_{\mathcal{D} (a, z, t)} - \left( \frac{\omega_h \cdot \nabla_h
    \omega_3}{f (\omega_3)^3} \right) v_3 d x_1 d x_2 \leqslant
    \int_{\mathcal{D} (a, z, t)} \left\{ \frac{| \nabla_h \omega_3 |^2}{4 f
    (\omega_3)^4} + C v_3^2 \right\} d x_1 d x_2. \label{eq:W-rhs-cauchy}
  \end{equation}
  Combining \eqref{eq:114-1}--\eqref{eq:W-rhs-cauchy} we
  obtain that, for $(z, t) \in \mathcal{P}_a$,
  \[ \partial_t W - \partial_z^2 W \leqslant \int_{\mathcal{D} (a, z, t)} (|
     \partial_3 v_3 | + Cv_3^2) d x_1 d x_2 + G^{O(G)} . \]
  Since the local energies of $v$ are bounded by $O(G)$ by
  \eqref{eq:local-energy-ineq}, we have
  \[ \left\| \int_{\mathcal{D} (a, z, t)} | \partial_3 v_3 | d x_1 d x_2
     \right\|_{{L^2_t}  L^2_z (\mathcal{P}_a)} + \left\| \int_{\mathcal{D} (a,
     z, t)} v_3^2 d x_1 d x_2 \right\|_{{L^{\infty}_t}  L^1_z (\mathcal{P}_a)}
     \leqslant O(G). \]
  Moreover, using Lemma \ref{thm:qrs}, we have pointwise boundary estimates
  for $W$ given by
  \[ \sup_{- a^2 < t < 0} \sup_{j = \pm} W (j a, t) + \sup_{- a < z < a} W (z,
     - a^2) \leqslant G^{O(G)} . \]
  Hence, by standard estimates for 1D heat equation we obtain the desired
  uniform bound
  \[ \sup_{(z, t) \in \mathcal{P}_a} W(z,t) \leqslant G^{O(G)} . \]
  
\end{proof}

Next, we derive the uniform Type I bounds and apply the blow-up
procedure to $v$, based on Lemma \ref{lem:flux}.

\begin{lemma}
  \label{lem:Type-I}Under the assumption of Lemma \ref{lem:flux}, $(v, p)$
  satisfies the uniform Type I condition \eqref{eq:unif-Type-I} in
  $\mathcal{Q} \left( \frac{1}{2} \right)$.
\end{lemma}

\begin{proof}
  As before we write
  \[ G = \int_{\mathcal{Q} (1)} \left( | v |^3 + | p |^{\frac{3}{2}} \right)
     {dxdt} + 2. \]
  We will show that for any $\frac{3}{2} < q < 2$,
  \begin{equation}
    \sup_{(x_0, t_0) \in \mathcal{Q} \left( \frac{1}{2} \right)} \sup_{0 < r <
    \frac{1}{10}} \sup_{- r^2 < t - t_0 < 0} r^{q - 3} \int_{\mathcal{B} (r;
    x_0, t_0)} |v|^q d x \le  C (q) G^{O(G)}, \label{eq:critical-v}
  \end{equation}
  which implies Lemma \ref{lem:Type-I} via Lemma \ref{lem:critical-Type-I}. To
  prove \eqref{eq:critical-v}, we rely on \eqref{eq:unif-omega} and the
  Biot-Savart law. It is sufficient to consider the case $(x_0, t_0) = (0, 0)$,
  since by translation and scaling the estimates for general $(x_0, t_0) \in \mathcal{Q}
  \left( \frac{1}{2} \right)$ can be similarly deduced. Let $a, \delta>0$ be the numbers from
  Lemma \ref{thm:qrs}. Consider a smooth spatial cut-off function $\phi (x)$
  satisfying $\phi = 1$ for $|x| < a - \delta$ and $\phi = 0$ for $|x| > a$.
  Let $\tilde{v} = v \phi$, then using the fact that $\Delta v = - \nabla \times \omega$, we have
  \begin{align*}
    \Delta \tilde{v} & = \Delta v \phi + 2 \nabla v \cdot \nabla \phi + v
    \Delta \phi\\
    & = - \nabla \times (\omega \phi) - \omega \times \nabla \phi + 2 \nabla
    v \cdot \nabla \phi + v \Delta \phi .
  \end{align*}
  
  Introducing the rescalings
  \[ v^{(r)} (x, t) = r v (r x, r^2 t), \quad \phi^{(r)} (x) = \phi (r x), \]
  \[ \tilde{v}^{(r)} (x, t) = r \tilde{v} (r x, r^2 t), \quad \omega^{(r)} (x,t) = r^2
     \omega (r x,  r^2 t), \]
  with $0<r<\frac{1}{10}$, there holds that
  \begin{equation}
    \Delta \tilde{v}^{(r)} = - \nabla \times (\omega^{(r)} \phi^{(r)}) -
    \omega^{(r)} \times \nabla \phi^{(r)} + 2 \nabla v^{(r)} \cdot \nabla
    \phi^{(r)} + v^{(r)} \Delta \phi^{(r)} . \label{eq:tilde-v-r}
  \end{equation}
  By \eqref{eq:reg-est-shell} we can bound the Newtonian potential of the ``cross terms'' as
  \begin{equation}
    \sup_{|x| < 1} | \Delta^{- 1} (-\omega^{(r)} \times \nabla \phi^{(r)} + 2
    \nabla v^{(r)} \cdot \nabla \phi^{(r)} + v^{(r)} \Delta \phi^{(r)}) | \le
    G^{O(G)} r \le G^{O(G)}.
  \end{equation}
  For the first term on the right hand side of \eqref{eq:tilde-v-r},  at any negative time $t$ we have
  \begin{eqnarray}
    |\Delta^{- 1} \nabla \times (\omega^{(r)} \phi^{(r)}) (x) | & \lesssim &
    \int_{\mathbb{R}^3} \frac{| \omega^{(r)} \phi^{(r)} | (y)}{|x - y|^2} d y
    \nonumber\\
    & = & \int_{|y| \le 2} \frac{| \omega^{(r)} | (y)}{|x - y|^2} d y +
    \int_{|y| > 2} \frac{| \omega^{(r)} \phi^{(r)} | (y)}{|x - y|^2} d y
    \nonumber\\
    & = : & J_1 + J_2. \label{eq:J-1-2} 
  \end{eqnarray}
  By Young's inequality for convolutions, for any $- 1 \le z \le 1$, we have
  \begin{eqnarray}
    &  & \left( \int_{|x_h | \le 1} |J_1 (x_1, x_2, z) |^q d x_h
    \right)^{ \frac{1}{q}} \nonumber\\
    & \le &\left( \sup_{| \zeta | \le 2}  \int_{|x_h | \le 2} | \omega^{(r)} (x_1,
    x_2, \zeta) | d x_1 d x_2 \right) \left( \int_{| \zeta | \le 2} \left( \int (|x_h |^2 +
    |z - \zeta |^2)^{- q} d x_h \right)^{ \frac{1}{q}} d \zeta \right)
    \nonumber\\
    & \le & \left(\sup_{| \zeta | \le 2 r}  \int_{|x_h | \le 2 r} | \omega
    (x_1, x_2, \zeta) |d x_1 d x_2 \right) \left( C \int_{| \zeta | \le 2} |z - \zeta |^{- 2 +
    \frac{2}{q}} d \zeta \right) \nonumber\\
    & \leqslant & C (q) G^{O(G)} . \label{eq:J-1} 
  \end{eqnarray}
  For $J_2$ on $|x| \le 1$, since $|x-y|$ is separated from $0$ we can use a more direct bound:
  \begin{eqnarray}
    \sup_{|x| \le 1} |J_2 (x) | & \le & C   \sup_{\zeta \in \mathbb{R}} 
    \int_{\mathbb{R}^2} | \omega^{(r)} (x_1, x_2, \zeta) \phi^{(r)} | d x_1 d
    x_2 \nonumber\\
    & \le & C \sup_{| \zeta | \le a}  \int_{|x_h | \le a} | \omega (x_1,
    x_2, \zeta) | d x_1 d x_2 \nonumber\\
    & \leqslant & G^{O(G)} . \label{eq:J-2} 
  \end{eqnarray}
  Combining \eqref{eq:J-1-2}--\eqref{eq:J-2}, we arrive at
  \begin{equation}
    \sup_{- 1 < t < 0} \sup_{- 1 < z < 1}  \int_{\mathcal{D} (1, z, t)}
    |v^{(r)} |^q d x_1 dx_2 \le  C (q) G^{O(G)},
  \end{equation}
  which implies \eqref{eq:critical-v} by rescaling and H{\"o}lder's
  inequality.
\end{proof}

Based on Lemma \ref{lem:Type-I}, we can apply a standard blow-up procedure to
the solution $v$ at the origin. Consider a sequence of rescaled solutions
\[ v^{(r_k)} (x, t) = r_k v (r_k x, r_k^2 t), \quad p^{(r_k)} (x, t) = r_k^2 p
   (r_k x, r_k^2 t), \]
with $r_k \searrow 0$ as $k \to + \infty$. By {\cite[Proposition 2]{SS}},
there exist subsequences of $v^{(r_k)}, p^{(r_k)}$ (for simplicity, we still
denote them by $v^{(r_k)}, p^{(r_k)}$) such that, for each $R > 0$,
\[ v^{(r_k)} \to \hat{v} \]
in $L^3_t L^3_x (\mathcal{Q} (R))$ and
\[ p^{(r_k)} \rightharpoonup \hat{p} \]
in $L^{ \frac{3}{2}}_t {L^{\frac{3}{2}}_x} (\mathcal{Q}
(R))$, where $(\hat{v}, \hat{p})$ is a suitable weak ancient solution living
on $\mathbb{R}^3 \times (- \infty, 0)$. The scale-invariant quantities for
$(\hat{v}, \hat{p})$ (we denote by $\hat{F}$, $\hat{A}$, $\hat{E}$, $\hat{D}$, see
\eqref{eq:FE}--\eqref{eq:0115-1}) are bounded globally in $\mathbb{R}^3
\times (- \infty, 0)$, \emph{i.e.},
\[ \sup_{(x_0, t_0) \in \mathbb{R}^3 \times (- \infty, 0)} \sup_{r > 0} 
   (\hat{F} + \hat{A} + \hat{E} + \hat{D}) (r; x_0, t_0) < + \infty . \]
Moreover, $O$ is a singular point of $v$ if and only if the origin is a
singular point of $\hat{v}$. Let $\omega^{(r_k)} = \nabla \times v^{(r_k)}$
and $\hat{\omega} = \nabla \times \hat{v}$. By the assumption
\eqref{eq:cone-3}, at any regular points of $v^{(r_k)}$ we have
\[ | \omega^{(r_k)} | \le C| \omega^{(r_k)}_3 | + M r_k^2 \]
and hence in the limit, the vorticity $\hat{\omega} = \nabla \times \hat{v}$ satisfies \eqref{eq:cone-3} with $M =
0$, i.e., at any regular points of $\hat{v}$, we have
\begin{equation}
  \quad | \hat{\omega} | \le C | \hat{\omega}_3 | . \label{eq:limit-M0-34}
\end{equation}
To be precise, in deriving \eqref{eq:limit-M0-34} we need the persistence of
singularity result (see {\cite[Lemma 2.1]{RS}}), which implies that the
convergence of $\omega^{(r_k)} \rightarrow \hat{\omega}$ is locally uniform at
regular points of $\hat{v}$. In summary, we have reduced the proof of Theorem
\ref{thm:main} to the following weaker result:

\begin{proposition}
  \label{prop:main}Let $v$ be a suitable weak solution to the Navier-Stokes
  equations in $\mathcal{Q} (1)$. Suppose that \eqref{eq:cone-3-0} holds at
  all regular points of $v$, then $v$ is regular at the origin.
\end{proposition}

\paragraph{Proof of Theorem \ref{thm:main} via Proposition \ref{prop:main}.}By
Proposition \ref{prop:main}, the limiting solution $\hat{v}$ considered above
is regular at the origin. Hence, by {\cite[Proposition 2]{SS}}, $v$ is also
regular at the origin.

\section{Local smallness of the absolute vorticity flux}\label{sec:4}

In this Section, we prove Proposition \ref{prop:main} by establishing a local
decay property for the absolute vorticity flux. This is the most delicate part of our approach, providing the crucial smallness to break the scaling. Here, we still rely heavily on the structural cancellations between the two nonlinear terms in the vorticity equation.

\begin{proof}[Proof of Proposition \ref{prop:main}]
Let $v$ be a suitable weak solution to the Navier-Stokes equations in
$\mathcal{Q} (1)$ satisfying \eqref{eq:cone-3-0} at all regular points. By the
reduction argument in the second paragraph of Section \ref{sec:3}, we can
assume in addition that $v$ is regular at all negative times. By Lemma
\ref{lem:Type-I}, we have
\begin{equation}
  K := 2 + \sup_{0 < r < 1} \frac{1}{r^2} \int_{\mathcal{Q} (r)} \left( |
  v |^3 + | p |^{\frac{3}{2}} \right) d x d t < + \infty . \label{eq:K}
\end{equation}
Define $\tilde{\omega} = \frac{\omega_3}{| \omega_3 |} \omega$, and recall from Section \ref{sec:3}
that \eqref{eq:cone-3-0} implies

\begin{equation}
  \nabla \tilde{\omega} = \frac{\omega_3}{| \omega_3 |} \nabla \omega
\end{equation}
and
\begin{equation}
  \nabla \cdot \tilde{\omega} = 0.
\end{equation}
Denote the local absolute vorticity flux through discs by
\[ \Gamma (r, z, t) = \int_{\mathcal{D} (r, z, t)} | \omega_3 (x_h, z, t) | d
   x_1 d x_2 . \]
By Lemma \ref{lem:flux} (with slightly different choice of constants), we have an upper bound
\begin{equation}
  \sup_{\mathcal{Q} \left( \frac{9}{10} \right)} \Gamma < K^{O(K)}.
  \label{eq:Gamma-upper}
\end{equation}
It is useful to think of $\Gamma$ as an axi-symmetric function in three space
dimensions, independent of the angular variable $\theta$. Note that $\Gamma$ is an analogue of the swirl function $r v^\theta$ extensively used in the study of axisymmetric solutions.

Multiplying the
vorticity equation \eqref{eq:omega-3} by $\frac{\omega_3}{| \omega_3 |}$ we
get
\begin{equation}
  \partial_t | \omega_3 | - \Delta |\omega_3| + v \cdot \nabla | \omega_3 | -
  \tilde{\omega} \cdot \nabla v_3 \leqslant 0 \label{eq:abs-omega-3}
\end{equation}
in the sense of distributions.  Integration by parts in the horizontal variables gives
\begin{eqnarray}
	&  & \int_{\mathcal{D} (r, z, t)} \Delta | \omega_3 | d x_1 d x_2 \nonumber\\
	& = & \partial_z^2 \Gamma + \int_{\mathcal{S} (r, z, t)} \partial_r |
	\omega_3 | d s \nonumber\\
	& = & \partial_z^2 \Gamma + \partial_r \int_{\mathcal{S} (r, z, t)} |
	\omega_3 | d s - \frac{1}{r} \int_{\mathcal{S} (r, z, t)} | \omega_3 | d s \nonumber\\
	& = & \partial_z^2 \Gamma + \partial_r^2 \Gamma - \frac{1}{r} \partial_r
	\Gamma,  \label{eq:0115-3}
\end{eqnarray}
and
\begin{eqnarray}
	&  & \int_{\mathcal{D} (r, z, t)} (v \cdot \nabla | \omega_3 | -
	\tilde{\omega} \cdot \nabla v_3) d x_1 d x_2 \nonumber \\
	& = & \int_{\mathcal{D} (r, z, t)} \large{( - \nabla_h \cdot v_h
		| \omega_3 | + v_3 \partial_3 | \omega_3 | + \nabla_h \cdot \tilde{\omega}_h
		v_3 - | \omega_3 | \partial_3 v_3 )\Large} d x_1 d x_2 \nonumber \\
	&  & + \int_{\mathcal{S} (r, z, t)} (v \cdot e_r | \omega_3 | -
	\tilde{\omega} \cdot e_r v_3) d s \nonumber \\
	& = & \int_{\mathcal{S} (r, z, t)} (v \cdot e_r | \omega_3 | -
	\tilde{\omega} \cdot e_r v_3) d s. \label{eq:0115-4}
\end{eqnarray}
 Hence, integrating \eqref{eq:abs-omega-3} in the horizontal variables and using \eqref{eq:0115-3}--\eqref{eq:0115-4}, we obtain that for $- 1 < t < 0$, $- 1 < z < 1$ and $0 < r < 1$,
\begin{equation}
  \partial_t \Gamma - \partial_z^2 \Gamma - \partial_r^2 \Gamma + \frac{1}{r}
  \partial_r \Gamma + \int_{\mathcal{S} (r, z, t)} (v \cdot e_r | \omega_3 | -
  \tilde{\omega} \cdot e_r v_3) d s \le 0. \label{eq:Gamma-30}
\end{equation}
Applying Lemma \ref{thm:qrs} to $v$, there exist numbers $a \in
\big(\frac{2}{3}, \frac{4}{5}\big)$ and $\delta \geqslant K^{-O(K)} > 0$ such that $v$ is
regular in the space-time closure of $\mathcal{Q} (a + \delta) \setminus
\mathcal{Q} (a - \delta)$ with the estimates
\begin{equation}
  \sup_{\mathcal{Q} (a + \delta) \setminus \mathcal{Q} (a - \delta)} (| v | +
  | \nabla v | + | \nabla^2 v |) < K^{O(K)}. \label{eq:reg-v-30}
\end{equation}
For simplicity of presentation, in the sequel we do not keep track of the exact dependence of constants on $K$. Consider the domains
\begin{equation}
  \mathcal{D}_1 = \{(r, z, t) : 0 < r < a + \delta, | z | < a + \delta, - (a +
  \delta)^2 < t < - (a - \delta)^2 \}
\end{equation}
and
\begin{equation}
  \mathcal{D}_2 = \{(r, z, t) : a - \delta < r < a + \delta, | z | < a +
  \delta, - (a + \delta)^2 < t < 0\} .
\end{equation}
For $(r,z,t) \in \mathcal{D}_1 \cup \mathcal{D}_2$, due to \eqref{eq:cone-3-0} and \eqref{eq:reg-v-30} we have
\begin{align*}
  \int_{\mathcal{S} (r, z, t)} (v \cdot e_r | \omega_3 | - \tilde{\omega}
  \cdot e_r v_3) {ds} \ & {\geqslant} - C (K)
  \int_{\mathcal{S} (r, z, t)} | \omega_3 | d s\\
  & = - {C (K)} \partial_r \Gamma,
\end{align*}
hence by \eqref{eq:Gamma-30},
\begin{equation}
  \partial_t \Gamma - \partial_z^2 \Gamma - \partial_r^2 \Gamma - \frac{C
  (K)}{r} \partial_r \Gamma \le 0 \label{eq:Gamma-34}
\end{equation}
in the sense of distributions.

\smallskip

Our next goal is to show that
\begin{equation}
  \lim_{r \rightarrow 0 +} \sup_{\mathcal{Q} (r)} \Gamma = 0.
  \label{eq:Gamma-decay}
\end{equation}
We prove \eqref{eq:Gamma-decay} by contradiction. Suppose that
\begin{equation}
  \sup_{\mathcal{Q} (r)} \Gamma \ge \kappa \label{eq:kappa}
\end{equation}
for some positive constant $\kappa$ uniformly in $r \in (0, 1)$. By
\eqref{eq:reg-v-30}, there exists $r_{\ast} = r_{\ast} (K, \kappa) > 0$ such
that
\begin{equation}
  \Gamma (r, z, t) < \frac{\kappa}{2} \label{eq:Gamma-kappa-2}
\end{equation}
for $r < r_{\ast}, |z| < a, - (a + \delta)^2 < t < - (a - \delta)^2$. Let
\[ V(r,z,t) = \sup_{\mathcal{Q} \left( \frac{9}{10} \right)} \Gamma - \Gamma(r,z,t), \]
then by \eqref{eq:Gamma-34} we have
\begin{equation}
  \partial_t V - \partial_z^2 V - \partial_r^2 V - \frac{C (K)}{r} \partial_r
  V \ge 0, \label{eq:V-34}
\end{equation}
and by \eqref{eq:kappa}--\eqref{eq:Gamma-kappa-2} we have
\[ V (r, z, t) > \frac{\kappa}{2} \]
for $r < r_{\ast}, |z| < a, - (a + \delta)^2 < t < - (a - \delta)^2$.
Applying Lemma \ref{lem:dnm} to $V$ in $\mathcal{D}_1$ viewed as a 3D axisymmetric domain, we obtain that
\[ V (r, z, t) > c_1 (K, \kappa) > 0 \]
for $a - \delta < r < a, |z| < a, - a^2 < t < - (a - \delta)^2$. Then,
applying Lemma \ref{lem:dnm} to $V$ in $\mathcal{D}_2$, we further obtain that
\begin{equation}
  V (r, z, t) > c_2 (K, \kappa) > 0 \label{eq:V-c2}
\end{equation}
for $a - \frac{\delta}{2} < r < a + \frac{\delta}{2}, | z | < a, - \frac{1}{4}
< t < 0$. In view of the monotonicity of $\Gamma$ in $r$, \eqref{eq:V-c2}
implies that
\begin{equation}
  \sup_{\mathcal{Q} \left( \frac{9}{20} \right)}   \Gamma < \sup_{\mathcal{Q}
  \left( \frac{9}{10} \right)} \Gamma - c_2 (K, \kappa) \begin{array}{l}
    <
  \end{array} (1 - \sigma (K, \kappa)) \sup_{\mathcal{Q} \left( \frac{9}{10}
  \right)} \Gamma \label{eq:Gamma-sigma}
\end{equation}
for some $\sigma (K, \kappa) > 0$. Note that in the second inequality of
\eqref{eq:Gamma-sigma}, we used the upper bound $\eqref{eq:Gamma-upper}$.

For any small $r > 0$, consider the rescaled solution $v^{(r)} (x, t) = r v
(r x, r^2 t)$ with corresponding $\Gamma^{(r)} (x, t) = \Gamma (r x, r^2 t)$.
Clearly, both \eqref{eq:cone-3-0} and \eqref{eq:K} still hold for $v^{(r)}$.
Repeating the arguments of deriving \eqref{eq:Gamma-sigma}, we get
\begin{equation}
  \sup_{\mathcal{Q} \left( \frac{9}{20} \right)} \Gamma^{(r)} \le (1 - \sigma
  (K, \kappa)) \sup_{\mathcal{Q} \left( \frac{9}{10} \right)} \Gamma^{(r)},
\end{equation}
or equivalently
\begin{equation}
  \sup_{\mathcal{Q} \left( \frac{9 r}{20} \right)} \Gamma \le (1 - \sigma (K,
  \kappa)) \sup_{\mathcal{Q} \left( \frac{9 r}{10} \right)} \Gamma .
  \label{eq:Gamma-sigma-iter}
\end{equation}
Iterating \eqref{eq:Gamma-sigma-iter}, we deduce that for $m = 1, 2, 3,
\cdots$,
\begin{equation}
  \sup_{\mathcal{Q} \left( 2^{- m} \times \frac{9}{10} \right)} \Gamma \le (1
  - \sigma (K, \kappa))^m \sup_{\mathcal{Q} \left( \frac{9}{10} \right)}
  \Gamma,
\end{equation}
which contradicts with \eqref{eq:kappa}. In conclusion, \eqref{eq:Gamma-decay}
holds true.

Repeating the proof of Lemma \ref{lem:Type-I} and using the smallness of $\Gamma$
instead of just boundedness, we deduce that
\[ \limsup_{r \to 0 +}  \frac{1}{r^2}  \int_{\mathcal{Q} (r)} (|v|^3 +
   |p|^{ \frac{3}{2}}) d x d t = 0. \]
Finally, standard $\varepsilon$-regularity theory  {\cite{CKN, Lin}} implies that the origin is a regular point of $v$.

\end{proof}

\section{Proofs of Corollaries \ref{cor:114} and \ref{thm:rot}}\label{sec:5}

\begin{proof}[Proof of Corollary \ref{cor:114}]
	If the origin is a regular point, then $\mathcal{I} = \emptyset$ since $|\omega|$ is locally bounded. On the other hand, if $\mathcal{I}$ misses a great circle on $\mathbb{S}^2$, then for some $r \in (0,1)$ and some $M>0$, $\mathcal{I}_{r,M}$ is contained in the intersection of a double cone with $\mathbb{S}^2$. By Theorem \ref{thm:main}, the solution will be regular at the origin.
\end{proof}

\begin{proof}[Proof of Corollary \ref{thm:rot}]
Let $v$ be a suitable weak solution to the Navier-Stokes equations in $Q (1)$
satisfying \eqref{eq:cone-t} at all regular points with $| \omega | \neq 0$.
Again, by the reduction argument in the second paragraph of Section
\ref{sec:3}, we can assume without loss of generality that $v$ is regular for
all negative times.

By \eqref{eq:cone-t}, for any time $- 1 < t < 0$, we can find some unit vector
$e$ such that
\begin{equation}
  | \omega | \leqslant C | \omega \cdot e | . \label{eq:cone-e}
\end{equation}
Indeed, if there exists an $x_0 \in B (1)$ such that $| \omega (x_0, t) | >
0$, then we can simply take $e = \xi (x_0, t)$. Otherwise, at time $t$ we have
$\omega \equiv 0$ and \eqref{eq:cone-e} trivially holds for an arbitrary $e \in \mathbb{S}^2$.
Take two unit vectors $e_a$, $e_b$ such that  \{$e_a, e_b, e$\}  forms an
orthonormal basis for $\mathbb{R}^3$ and denote the corresponding spatial
coordinates by
\begin{equation}
  y_1 = x \cdot e_a, \quad y_2 = x \cdot e_b, \quad y_3 = x \cdot e.
  \label{eq:def-y}
\end{equation}
Recall that in Section \ref{sec:3}, we proved the critical bound
\eqref{eq:unif-omega} under the assumption \eqref{eq:cone-3-0}. Using the same
argument (in $y$-coordinates) in the time $t$ slice, with \eqref{eq:cone-e}
instead of \eqref{eq:cone-3-0}, we get
\begin{equation}
  \sup_{- \frac{1}{2} < y_3 < \frac{1}{2}} \int_{\sqrt{y_1^2 + y_2^2} <
  \frac{1}{2}} | \omega | d y_1 d y_2 < C (G) . \label{eq:unif-omega-2}
\end{equation}
Note that the right hand side of \eqref{eq:unif-omega-2} is time-independent.
Using \eqref{eq:unif-omega-2} and repeating the arguments for Lemma
\ref{lem:Type-I} (using the $y$-coordinates), we deduce that $v$ satisfies the uniform Type I condition in $Q\big( \frac12 \big)$.

Consider the following two cases:

\smallskip
\emph{Case 1}. $\omega \equiv 0 \ \emph{a.e.}$ for $|x|<1$ and $t = 0$.
\smallskip

Consider a regular parabolic shell $Q(a+\delta) \setminus Q(a-\delta)$ guaranteed by Lemma \ref{thm:qrs}.  By continuity of $\omega$ within the regular shell together with the assumption of this case,  for any $\varepsilon>0$ there exists $T(\varepsilon) \in (-1,0)$ such that
\begin{equation} \label{eq:1026-1}
	|\omega| < \varepsilon
\end{equation}
pointwisely within the set $\left(B(a+\delta) \setminus B(a-\delta)\right) \times (T(\varepsilon),0)$. Repeating the proof of \eqref{eq:unif-omega-2} and using \eqref{eq:1026-1}, we deduce that
\begin{equation}
	\sup_{- \frac{1}{2} < y_3 < \frac{1}{2}} \int_{\sqrt{y_1^2 + y_2^2} <
		\frac{1}{2}} | \omega | d y_1 d y_2 < C \varepsilon, \label{eq:unif-omega-3}
\end{equation}
for $t \in (T(\varepsilon), 0)$. Hence, $v$ satisfies the Type I condition with smallness in $Q(|T(\varepsilon)|^\frac12)$, which leads to regularity at the origin if $\varepsilon$ is chosen small enough.

\smallskip
\emph{Case 2}. There exists some $x_{\ast} \in B(1)$ such that $(x_{\ast}, 0)$ is regular point and $| {\omega} (x_{\ast}, 0) | \neq 0$.

\smallskip
Let $e_{\ast} = { {\xi} (x_{\ast},
0) }$. By continuity, in a small parabolic neighbourbood of $(x_{\ast}, 0)$,
$|{\omega} |$ is separated from $0$ and ${\xi}$ is close to
$e_{\ast}$. By condition \eqref{eq:cone-t}, we obtain that \eqref{eq:cone}
holds (for some new $\delta > 0$) with $e = e_{\ast}$ in $\mathbb{R}^3 \times
(- h, 0)$ for $h$ sufficiently small. By Theorem \ref{thm:main}, we deduce
that ${v}$ is regular at the origin.

\medskip

In conclusion, $v$ must be regular at the origin. Hence, Corollary
\ref{thm:rot} is proved.
\end{proof}

\section{On the direction of velocity} \label{sec:last}

We can also prove a simple result on the direction of velocity, in the axisymmetric case. The extension of Proposition \ref{prop:115-000} to the general non-axisymmetric case would imply that near a potential singularity, the range of $\pm\frac{v}{|v|}$ is dense on the unit sphere.

\begin{proposition} \label{prop:115-000}
	Let $(v,p)$ be an axisymmetric suitable weak solution to the Navier-Stokes equations in $\mathcal{Q}(1)$. If for some $\delta, M>0$, there holds either $|v| \le M$ or
	\begin{equation} \label{eq:115-002}
		\left|\frac{v}{|v|} \cdot e_3\right| \le 1-\delta,
	\end{equation}
	then $v$ is regular at the origin.
\end{proposition}

\begin{proof}
	We use the standard cylindrical coordinates $(r, \theta, z)$ with $r=|x_h|$ and $z = x_3$. The condition \eqref{eq:115-002} is equivalent to
	\begin{equation}
		|v_z| \le C (1 + |v_\theta| + |v_r|).
	\end{equation}
	 By the second paragraph of Section \ref{sec:3}, we can assume that $v$ is regular for negative times. Consider any fixed time $-\frac12<t<0$. Let $\psi$ be the stream function defined by
	\begin{equation}
		\partial_r \psi = r v_z, \quad \partial_z \psi = -r v_r, \quad \psi|_{r=0} = 0
	\end{equation}
	Then, using the boundedness of $\Gamma = rv_\theta$ (see, \emph{e.g.}, \cite[Lemma 18]{LR}), we deduce that
	\begin{equation}
		|\partial_r \psi| \le C_1 + C_2 |\partial_z \psi|.
	\end{equation}
	Without loss of generality, suppose that $C_1, C_2 > 10$. Starting at any point $(r_0, \theta,z_0) \in \mathcal{B}\big(\varepsilon\big)$ with $\varepsilon = (100 C_2)^{-1}$, we perform the following exploration procedure:
	\begin{enumerate}[(i)]
		\item If $|\partial_r \psi| > 2C_1$, then we follow the level set of $\psi$ in the direction of decreasing $r$, until we reach $r=0$ and stop, or we reach a point where $|\partial_r \psi| \le 2C_1-1$ and perform  (ii). In this step, we always have
			$$|\partial_r \psi| \le 3C_2 |\partial_z \psi|,$$
		which ensures that we stay in $\mathcal{B}(1)$.
		
		 \item if $|\partial_r \psi| \le 2C_1$, then we simply move with $z$ fixed and $r$ decreasing, until we reach $r=0$ and stop, or we reach a point where $|\partial_r \psi| \ge 2C_1+1$ and perform (i).
	\end{enumerate}
	At the end of this procedure, we must end up with $r=0$, $\psi=0$. Keeping track of $\psi$ along our path, we obtain that
	$$|\psi(r_0, z_0)| \le 3C_1 r_0.$$ 
	This means that $v$ satisfies a critical estimate \cite{LZ}, and as a consequence, must be regular near the origin.
	\end{proof}

\

\paragraph{Acknowledgement.} Z.L. is in part supported by NSFC (No.
12431007), Sino-German Center (No. M-0548) and New Cornerstone Science Foundation
through the XPLORER PRIZE. X.R. was supported by the China Postdoctoral Science Foundation under Grant Number BX20230019 and the National Key R\&D Program of China (No.2023YFA1010700). G.T. was supported in part
by NSFC No.11890660 \& No.12341105, MOST No.2020YFA0712800.

\bibliographystyle{plain}
\bibliography{lrt}

\end{document}